\documentclass[english,twoside,a4paper,reqno,11pt]{amsart}
\usepackage{amsfonts, amsbsy, amsmath, amsthm, amssymb, latexsym, verbatim, enumerate}
\usepackage{mathrsfs}
\usepackage[top=30mm,right=30mm,bottom=30mm,left=30mm]{geometry}

\usepackage{bm}
\usepackage[pdftex]{color,graphicx}

\makeatletter
\newcommand{\imod}[1]{\allowbreak\mkern4mu({\operator@font mod}\,\,#1)}
\makeatother

\usepackage{hyperref}
\usepackage[T1]{fontenc}
\usepackage[latin9]{inputenc}
\usepackage{textcomp}
\usepackage{amsmath}
\usepackage{amsthm}
\usepackage{amssymb}

\makeatletter
\theoremstyle{plain}
\newtheorem{thm}{\protect\theoremname}[section]
\theoremstyle{definition}
\newtheorem{defn}[thm]{\protect\definitionname}
\theoremstyle{plain}
\newtheorem{cor}[thm]{\protect\corollaryname}
\theoremstyle{plain}
\newtheorem{lem}[thm]{\protect\lemmaname}
\theoremstyle{plain}
\newtheorem{prop}[thm]{\protect\propositionname}
\theoremstyle{remark}
\newtheorem*{rem*}{\protect\remarkname}

\usepackage{youngtab} 
\usepackage{mathrsfs}

\usepackage{pdflscape}
\date{}
\usepackage{changepage}

\usepackage{amsthm}

\theoremstyle{plain}
\newtheorem{mainthm}{Theorem}
\newcommand{\irr}{\text{Irr}(G)}

\newcommand{\ff}[1][q]{\mathbb{F}_{#1}}

\makeatother

\usepackage{babel}
\providecommand{\corollaryname}{Corollary}
\providecommand{\definitionname}{Definition}
\providecommand{\lemmaname}{Lemma}
\providecommand{\propositionname}{Proposition}
\providecommand{\remarkname}{Remark}
\providecommand{\theoremname}{Theorem}

\begin{document}
	
	\author{Alexander J. Malcolm}
	
		\address{Alexander J. Malcolm, School of Mathematics, University of Bristol, Bristol, BS8 1UG, UK, and the Heilbronn Institute for Mathematical Research, Bristol, UK.}
	\email{alex.malcolm@bristol.ac.uk}
	\title{On the $p$-width of finite simple groups}
	
	%\keywords{...}
	\begin{abstract}
	In this paper we measure how efficiently a finite simple group $G$ is generated by its elements of order $p$, where $p$ is a fixed prime.  This measure, known as the $p$-width of $G$,  is the minimal $k\in \mathbb{N}$ such that any $g\in G$ can be written as a product of at most $k$ elements of order $p$. Using primarily character theoretic methods, we sharply bound the $p$-width of some low rank families of Lie type groups, as well as the simple alternating and sporadic groups.
	
	\end{abstract}
%	\begin{abstract}
%	We present the first step in classifying the $p$-width of non-abelian finite simple groups for odd primes $p$. Recall that for a finite group $G$ with order divisible by prime $p$, we say that the $p$-width of $G$ is the minimal integer $k$ such that every element of $G$ is a product of at most $k$ elements of order $p$. Note that this is more restrictive than the case of $p$-elements, that is elements of order $p^l$ for some $l$.
%	In this work we focus on the natural first cases: the alternating groups, the sporadic groups, and the simple  lie type groups of rank 1. These base cases will be essential for the more difficult project of all finite simple groups where an inductive involvement involving low rank subgroups is likely to be necessary. See for example, \cite{AJM18} where this method is used to solve the involution width problem.
%	\end{abstract}
%	

	\date{\today}
	\maketitle
\section{Introduction}
The generation of finite simple groups has been a thriving area of research for many years. Since it was established that every finite simple group is generated by a pair of elements (Steinberg \cite{Stein} and Aschbacher-Guralnick \cite{AG}) many interesting refinements have been investigated: for instance determining the existence of generators with prescribed orders (see \cite{b_survey} and the references therein). 

More recently the notion of width has provided an additional perspective with which to view generation, giving a measure of how efficiently a given subset generates the group. More formally, suppose $G$ is a finite group generated by a subset $S$,  and for $g\in G$, let $w(g)=\min \{k \mid g=s_1\cdots s_k,\,s_i \in S\}$. Then the width is defined
\[
\text{width}(G,S)=\max_{g\in G} w(g).
\]
Equivalently, $\text{width}(G,S)$ is the diameter of $\Gamma(G,S)$, the Cayley graph of $G$ with respect to $S$. 
%In essence, $\text{width(G,S)}$ is a measure of the speed with which $S$ generates the group.
Now if $G$ is a non-abelian finite simple group then any nontrivial normal subset $S$ will generate $G$, and a number of interesting cases have been considered in the recent literature. For example, \cite{ore} settles the long-standing conjecture of Ore that every element is a commutator (i.e. the \emph{commutator width} is one). Furthermore it is shown in  \cite{squares} that $G$ has width two with respect to its set of squares.

In this work we consider the width of (non-abelian) finite simple groups with respect to elements of a fixed prime order: suppose that a prime $p$ divides $\vert G \vert$ and denote the set of order $p$ elements by $I_p(G)$. Then we consider $w_p(G):=\text{width}(G,I_p(G))$ and call this the \emph{$p$-width} of $G$.

The case where $p=2$ (also known as the \emph{involution width}) has received considerable attention in previous work. For example there exists a classification of the so-called strongly real finite simple groups i.e. those such that $w_2(G)=2$ (see \cite{GV} and the references therein). In fact $w_2(G)\leq 4$ for all finite simple groups \cite[Thm. 1]{AJM18} and one can show that this bound is sharp for a number of infinite families (for example $PSL_n(q)$ such that $n,q\geq 6$ and $\gcd(n,q-1)=1$ \cite{KN}).

In contrast, existing work on the $p$-width for odd primes is very limited. In \cite{GT}, the larger generating set $P$ of all $p$-elements is considered and it's shown that $\text{width}(G,P)\leq 70$ for all finite simple groups $G$ \cite[3.8]{GT}. We show that this bound can be reduced significantly in many cases (even on restriction to only prime order elements).

%In this work we consider the $p$-width of finite simple groups, that is $\text{width(G,I_p(G))}$ where $I_p(G)=$ with respect to set of prime order elements in finite simple groups: suppose that $p$ divides then  

%Any subset of elements closed under conjugation will generate a non-abelian finite simple group 
%The width of finite simple groups with respect to a number of interesting generating sets has been considered in the recent literature. For example
%A number of interesting questions concerning the width of finite simple groups have been addressed in the recent literature. For example \cite{Ore} settles the longstanding conjecture of Ore 

\begin{mainthm}\label{thm:MainThm}
	Let $G$ be a finite simple group that is either an alternating, sporadic or rank 1 group of Lie type, and let $p$ be an odd prime dividing $\vert G \vert$. Then $w_p(G)\leq 3$.
\end{mainthm}
We give more precise results for almost all groups under consideration. For example, in Sections \ref{sec:LieType} and \ref{sec:sporadic} we describe the exact $p$-width for all rank 1 groups of Lie type and sporadic groups. In most of these cases, the $p$-width of $G$ is actually 2 (see Propositions \ref{prop:PSL2 p width}--\ref{prop:ree} and Theorem \ref{thm: sporadic p width}). Nonetheless, Theorem \ref{thm:MainThm} is sharp: in Section \ref{sec:Alt} we describe a family of alternating groups of $p$-width 3, for any prime $p \neq 3$. Furthermore,  $w_3(PSU_3(q))=3$ if and only if $q\equiv 2 \mod 3$ (see Prop. \ref{prop:PSU3} below).

It is interesting to compare Theorem \ref{thm:MainThm} to the case where $p=2$. Whereas $w_2(G)=2$ for a sporadic simple group $G$ if and only if $G\in \{J_1, J_2\}$, $w_p(G)=2$ for almost all pairs of odd $p$ and sporadic $G$. In fact, $w_p(G)\leq w_2(G)$ for all odd primes $p$, and all groups considered in Theorem \ref{thm:MainThm}, with the exception of $(G,p)=(PSL_2(2^{2l+1}),3)$. 

In \cite{AJM18}, one technique for bounding the involution width of $G$ is to reduce to a collection of  Lie-type subgroups of low rank where the result is known. We expect that similar methods will be important in classifying the $p$-width of all finite simple groups, and that Theorem \ref{thm:MainThm} will be useful in this endeavour. 

Character-theoretic tools are also important in examining problems of width (see Section \ref{sec:LieType} below), but complete character tables are unknown for groups of Lie type of rank greater than 7. We expect this to be substantially more restrictive in finding sharp bounds for the $p$-width for odd primes than for $p=2$. 

Now for any prime $p$, there exists a constant $c$ (independent of $p$) such that $w_p(G)\leq c$ for all finite simple groups $G$ with $I_p(G)\neq \emptyset$  \cite[Cor. 1.5]{Liebeck and Shalev}. Theorem \ref{thm:MainThm} therefore motivates a number of natural questions: 

\textbf{Question 1:} For a given odd $p$, do there exist groups of Lie type of unbounded rank and $p$-width three?

\textbf{Question 2:} Does there exist any finite simple group $G$ and odd prime $p$ such that $w_p(G)>3$? 

We address the former question in Section \ref{sec:final} where we consider classical groups and a connection with Artin's conjecture, while the latter question will be examined in future work. Computational evidence and the results of this paper suggest a negative answer to Question 2.

\textbf{Remarks on notation and structure of the paper:}
Throughout, $p$ will denote a fixed prime divisor of $\vert G \vert$. The $p$-width of an element $g\in G$, denoted $w_p(g)$, is then the minimal $k\in \mathbb{N}$ such that $g\in I_p(G)^k =\{x_1\cdots x_k\,\vert \, x_i\in I_p(G)\}$. It follows that $w_p(G)=\max_{g\in G}w_p(g).$ Lastly, we denote the $G$-conjugacy class of an element $g$ by $g^G$.

In Section \ref{sec:Alt} we bound the $p$-width of the alternating groups $A_n$, using existing work of Dvir \cite{Dvir} and Bertram \cite{Ber72}, as well as a careful consideration of the case $p=3$ (see Lemma \ref{lem:An p=3}). Furthermore, for fixed $p$ and increasing $n$, we examine the asymptotic behaviour of  $w_p(A_n)$ using work of Larsen and Shalev \cite{LaSh}.  The bulk of this paper is then dedicated to the rank 1 Lie type groups and the sporadic groups, in Sections \ref{sec:LieType} and \ref{sec:sporadic} respectively. Here we translate a study of products of conjugacy classes in $I_p(G)$ into a question concerning irreducible characters via a classical result of Frobenius (see Theorem \ref{thm:cmc}). We then utilise known character tables of these groups to sharply bound the $p$-width. Finally,  in Section \ref{sec:final} we remark on groups of Lie type of unbounded rank.

\section{Alternating groups}\label{sec:Alt}
In this Section we consider the simple alternating groups $A_n, n\geq 5$. Our first result Theorem \ref{thm:An}, bounds the $p$-width for any choice of $n$ and $p$. However for a fixed odd prime $p$, we also study the asymptotic behaviour of the $p$-width, showing in Theorem \ref{thm:alternating-asymptotic} that $w_p(A_n)=2$ for all but finitely many values of $n$.

\begin{thm}\label{thm:An}
	Fix a prime $p$. The $p$-width of $A_n, (n\geq p)$ is at most three.
\end{thm}
The bound in Theorem \ref{thm:An} is
sharp for each choice of $p\neq 3$:  for example, if $p=2$ then all but the strongly real groups $A_{5},A_{6},A_{10},A_{14}$
have $2$-width three \cite[2.3]{AJM18}. Furthermore, if $p\geq5$ then \cite[Thm. 1 and 2]{Ber72} imply that $A_n\neq I_p(A_n)^2$ for the non-empty range $(4p+3)/3<n<2p$. The case $p=3$ is an interesting exception which we examine separately in Lemma \ref{lem:An p=3}.

As remarked earlier, the involution width of finite simple groups is well understood and \cite[2.3]{AJM18} proves Theorem \ref{thm:An} in the case $p=2$. We therefore restrict our attention to odd primes. But we note that this has also appeared in the literature before: the $p$-width of the alternating groups is studied by this author in the unpublished note \cite{AJMp}, and Theorem \ref{thm:An} is proved there by direct combinatorial methods.  Here we present a greatly simplified proof of Theorem \ref{thm:An} using existing work of Dvir \cite{Dvir}.

First let's establish some notation. 

	\begin{defn}\label{def:Initial definitions of An}Let $g\in S_{n}.$ Denote the
			support of $g$ by
			$\mu(g)$, the number of non-trivial cycles in the decomposition of $g$ by $c^{*}(g)$ and the number of fixed points by $\text{fix}(g)$. Lastly, set $r(g):=|\mu(g)|-c^{*}(g).$  This notation extends naturally to conjugacy classes.
\end{defn}

Before considering the general case, we restrict our attention to $p=3$.

\begin{lem}\label{lem:An p=3}
	The $3$-width of $A_n$, $n\geq 5$ is two. 
\end{lem}
	\begin{proof}
	Firstly observe that we can decompose an  $l$-cycle ($l\geq 4$), into a product of a $3$-cycle and  an $(l-2)$-cycle, in the following two ways:
	\begin{equation}\label{eq:bsl}
	(1,\dots,l)=(1,2,3)(4,\dots,l,1) 
	\end{equation}
\begin{equation}\label{eq:bsr}
	(1,\dots,l)=(l-2,1,2,\dots,{l-3})(1,l-1,l).
\end{equation}
We refer to the decomposition (\ref{eq:bsl}) as "breaking and permuting left" as we have broken off a $3$-cycle from our initial $l$-cycle, and then permuted  the entries in the remaining $(l-2)$-cycle one step to the left. Similarly, we refer to (\ref{eq:bsr}) as "breaking and permuting right".
In particular after breaking and permuting left and then right we produce 
\[
(1,\dots,l)=(1,2,3)\cdot (l-1,4,\dots, l-2)\cdot (4,l,1).
\]
Now suppose that $l\geq 5$ is odd. It is clear from the above that we can write $g=(1,\dots,l)$ as a product in $I_3(G)^2$ by progressively breaking and permuting left then right, and collecting the product of disjoint $3$-cycles on either side. As $l$ is odd, the middle cycle will reduce in length by 2 at each step, reducing to a $3$-cycle itself.  In particular, letting $m=\lfloor(l-1)/2 \rfloor$, we write $g$ as the following product  of $m$ $3$-cycles:

\[
(1,\dots,l)=\prod_{i=0}^{\lceil m/2\rceil-1}(l+1-2i,2i+2, 2i+3) \cdot \prod_{i=0}^{\lfloor m/2\rfloor -1}(4+2i,l-2i,l-2i+1),
\]
where $l+1$ is to be read as $1$. We denote this product by $(1,\dots,l)=x\cdot y$.

For example, if $l=11$, then breaking and permuting three times to left, and twice to the right yields
\begin{eqnarray*}
(1,\dots,11) &=& (1,2,3)\cdot(10,4,\dots,9)\cdot (4,11,1) \\
&=& (1,2,3)(10,4,5)(8,6,7)\cdot (6,9,10)(4,11,1).
\end{eqnarray*}

To check that $x\in I_3(G)$ it suffices to show that  $\vert \mu(x)\vert=3 \lceil m/2\rceil$. Well suppose for a contradiction that there exists $i\leq j \leq \lceil m/2\rceil-1$ such that $l+1-2i = 2j+3$. Rearranging then gives that $l\leq 4j+2$. But
\[
4j+2\leq 4(\lceil m/2\rceil-1)+2\leq 2m\leq l-1,
\] 
giving us our contradiction. The remaining case follows similarly, as does the proof that $y\in I_3(G)$. Hence any cycle of odd length $l\geq 5$  has $3$-width two. 
 
Next consider the product of two disjoint cycles $g_1g_2$ of even lengths $l_1$ and $l_2$. Firstly  we apply the same process as above to $g_1$, breaking and permuting on the left then right until we have broken off $(l_1/2-1)$ $3$-cycles and a $2$-cycle remains in the middle. In particular letting  $l_1/2-1=m_1$ then
\begin{eqnarray*}
(1,\dots,l_1) &=&\prod_{i=0}^{\lceil m_1/2\rceil -1}(l_1+1-2i,2i+2, 2i+3) \\ & & \cdot (2 \lfloor (m_1+1)/2 \rfloor+2, 2 \lfloor (m_1+1)/2 \ \rfloor+3  ) \\ & & \cdot \prod_{i=0}^{\lfloor m_1/2\rfloor -1}(4+2i,l_1-2i,l_1-2i+1).
\end{eqnarray*}
For ease we relabel this product $(1,\dots,l_1)=x_1\cdot (a_1,a_1+1) \cdot y_1$ and note that 
if $m_1$ is odd, then $a_1 \notin \mu(x_1)\cup \mu(y_1)$ but $a_1+1\in \mu(x_1)\backslash \mu(y_1)$. On the other hand, when $m_1$ is even,  $a_1+1 \notin \mu(x_1)\cup \mu(y_1)$ but $a_1\in \mu(y_1)\backslash \mu(x_1)$. In other words, the set $\{a_1,a_1+1\}$ intersects non-trivially with the support of exactly one of $x_1$ or $y_1$ and this intersection has size one. Furthermore, the non-trivial intersection is determined by whether our final step was to break and permute to the left or to the right (determined by the parity of $m_1$).

Now we constructed this decomposition of $(1,\dots,l_1)$ by breaking to the left as the first step. But alternatively, breaking right first would produce a similar decomposition $(1,\dots,l_1)=x_1'(a_1',a_1'+1)y_1'$ such that $x_1',y_1'\in I_3(G)$ -- we leave the details of this to the reader. Crucially, this would then reverse which of $x_1'$ on the left, or $y_1'$ on the right, intersects non-trivially with the central $2$-cycle.

Applying this process to $g_1g_2$ (and recalling that $\mu(g_1)\cap \mu(g_2) = \emptyset$) yields
\begin{eqnarray*}
g_1g_2&=&x_1(a_1,a_1+1)y_1\cdot x_2(a_2,a_2+1)y_2 \\ &=&x_1x_2 \cdot (a_1,a_1+1)(a_2,a_2+1) \cdot y_1y_2,
\end{eqnarray*}
chosen such that $a_1,a_2,a_2+1 \notin \mu(x_1x_2)$ and $a_1,a_1+1,a_2+1 \notin \mu(y_1y_2)$.
But 
\[
(a_1,a_1+1)(a_2,a_2+1)=(a_1,a_2+1,a_2)(a_1,a_2+1,a_1+1).
\]
Hence $x_1x_2(a_1,a_2+1,a_2)$ and $(a_1,a_2+1,a_1+1)y_1y_2$ are elements of order 3, and $w_3(g_1g_2)=2$ as required.

So we have shown that $w_3(g)=2$ for  $g\in A_n$ that is either a single cycle of odd length at least 5, or a pair of even-length cycles. But any $h\in A_n$ can be written as a disjoint product of such elements plus a number of $3$-cycles and so the Lemma follows.
\end{proof}
Let's return to the general case of any odd prime $p$.  The crucial result of Dvir is the following:
\begin{thm} \cite[Thm. 10.2]{Dvir} \label{thm: Dvir Thm}
	Let $C\subset A_{n}$ be a conjugacy class not containing $2^k$-cycles for $k>1$ and such that $r(C)\geq \frac{1}{2}(n-1)$.  Then $C^3=A_{n}$. 
\end{thm}

\begin{cor} \label{cor:Dvir cor}
	Fix a prime $p\geq 3$ and suppose that $n\geq kp$ for some $k\in \mathbb{N}$. Let $C \subset A_{n}$ be a conjugacy class of elements of cycle type $(p^k,1^{n-kp})$. If $C^{3}\neq A_{n}$, then $n\geq (k+1)p$.
	
\end{cor}

\begin{proof}
	Firstly note that $r(C)=k(p-1)$. So if $k=1$ then (by Theorem \ref{thm: Dvir Thm}) the assumption that $C^{3}\neq A_{n}$ implies that $(n-1)/2>p-1$, which holds if and only if $n>2p-1$.
	Now assume that $k \geq 2$ and suppose for a contradiction that $n<(k+1)p$. Well again by Theorem \ref{thm: Dvir Thm}, the assumption that  $C^{3}\neq A_{n}$ implies that $(n-1)/2>k(p-1)$. Hence \[
	\frac{(k+1)p-1}{2}>k(p-1), \]
	and this reduces to $2k-1>p(k-1)$. But as $p \geq 3$, it follows that $p(k-1)\geq 3k-3$  and we have a contradiction.
\end{proof}

\textbf{Proof of Theorem \ref{thm:An} for $p\geq 3$:} This is immediate from Corollary \ref{cor:Dvir cor} as $n$ is finite. \qedsymbol

Now that we have bounded the $p$-width of $A_n$, it is interesting to ask whether this bound is sharp asymptotically. Similar asymptotic problems in symmetric groups have been considered by Larsen and Shalev \cite{LaSh}. In particular, the following result will be useful

\begin{thm}\label{thm:larsen}\cite[Thm. 1.10]{LaSh}
For all $\epsilon > 0$ there exists $N_0$ such that for all $n\geq N_0$, if $g \in S_n$ satisfies both
\begin{center} \begin{enumerate}
	\item $c^{*}(g)+\text{fix}(g)< (\frac{1}{4}-\epsilon) n$
	\item $\text{fix}(g^2)<n^{1/4-\epsilon}$
\end{enumerate} \end{center}
%\begin{equation}
%c^{*}(g)+\text{fix}(g)< (\frac{1}{4}-\epsilon) n 
% \\
%\text{ and } \text{fix}(g^2)<n^{1/4-\epsilon}, \label{eq:conditions2}
%\end{equation}
then $(g^{S_n})^2=A_n$.	\end{thm}
\begin{thm} \label{thm:alternating-asymptotic}	Fix $p\geq 3$. There exists $N$ such that for all $n> N$, $w_p(A_n)=2$.
\end{thm}
\begin{proof}
	As the case where $p=3$ is immediate from Lemma \ref{lem:An p=3}, we restrict our attention to $p\geq 5.$
	Fix $\epsilon = 1/p^2$ and let $N_0$ be the constant given by Theorem \ref{thm:larsen}. It suffices to show that for sufficiently large $n\geq N_0$, there exists an element $g \in I_p(A_n)$ satisfying conditions (1) and (2) of Theorem \ref{thm:larsen}.
	
	Well clearly if $n\geq p$ then $I_p(A_n)\neq \emptyset$, so assuming this, let $g \in I_p(A_n)$ be any element consisting of $\lfloor n/p \rfloor$ disjoint $p$-cycles.  Evidently \[c^{*}(g)+\text{fix}(g) \leq n/p+(p-1).\] 
	It then follows that if \[n>N_1:=\frac{4p^3-4p^2}{p^2-4p-4},\] then 
	\[
	c^{*}(g)+\text{fix}(g)< (\frac{1}{4}-\frac{1}{p^2}) n.\]
	Next we note that $\text{fix}(g^2)=\text{fix}(g)\leq p-1$ and that $n^{1/4-1/p^2}$ is strictly increasing in $n$. Hence there exists $N_2$ such that $\text{fix}(g^2)<n^{1/4-1/p^2}$ for all $n>N_2$.
 So letting \[
	N:=\max\{ N_0,N_1,N_2\}
	\]
it is clear that for all $n>N$, our chosen element $g$ satisfies conditions (1) and (2) and we are done by Theorem \ref{thm:larsen}.
\end{proof}
%\begin{defn}
%	For $g \in S_n$, let $f_k(g)$ denote the number of cycles of length $k$ in the decomposition of $g$. Furthermore let $\Sigma_k(g)$ denote the union of all $g$-orbits of length at most $k$ so that 
%	\[
%	\vert \Sigma_k(g)\vert=\sum_{i=1}^k if_i(g).	\]
%	Lastly, we define 
%	\[
%	E(g):=\frac{1}{\log n}\sum_{k\geq 1} \frac{\log^{+}\vert \Sigma_k \vert-\log^{+}\vert \Sigma_{k-1} \vert}{k},
%	\]
%	where we write $\log^+ x$ for $\max (\log x, 0)$.
%\end{defn}
%
%\begin{proof}
%	Let's assume that $n=ap+b$ where $0 \leq b <p$, and so there exists a class $C\subset I_p(A_n)$ with representative $g$ consisting of $a$ disjoint $p$-cycles. Note  that 
%	 \begin{equation}
%	\Sigma_k(g)=
%	\begin{cases}
%	b & \text{if } k<p,\\
%	n & \text{if } k\geq p. 
%	\end{cases}
%	\end{equation}
%	Furthermore,
%	\[
%	E(g)=\frac{\log n +b(p-1)}{p}.
%	\]
%	 \begin{equation}
%	E(g)=
%	\begin{cases}
%	b & \text{if } k<p,\\
%	n & \text{if } k\geq p. 
%	\end{cases}
%	\end{equation}
%\end{proof}
\section{Low rank groups of Lie type}\label{sec:LieType}
In this section we consider the $p$-width of low-rank groups of Lie type. This includes an examination of all groups of (twisted) rank 1. We recall that the case where $p=2$ is investigated in \cite{AJM18} and so we restrict our attention to odd primes.

Let 
\[
\mathcal{A}:= \{PSL_2(q), PSL_3^{\epsilon}(q), \,{}^{2}B_{2}(2^{2n+1}), {}^{2}G_{2}(3^{2n+1}), n\geq 1\}.
\]

\begin{thm}\label{thm:MainThmLieType}
	Let $G \in \mathcal{A}$ and let $p$ be an odd prime dividing $\vert G \vert$. Then $w_p(G)\leq 3$.
\end{thm}
Width questions are often reduced to studying the product of a given collection of conjugacy classes and that is our general strategy here: given $g \in G$, we show that $g\in C_1\cdots C_k$ for some collection of classes $C_1,\dots,C_k\subseteq I_p(G)$. From this it is then immediate that $w_p(g)\leq k$.
The character theory of $G$ plays an important role via the following well-known result.

\begin{thm}\label{thm:cmc}\cite[$\mathsection1$, Thm 10.1]{AH} Let $G$ be a finite group, $C_1,\dots,C_k$ be conjugacy classes of $G$ with representatives $g_1,\dots,g_k$, and let $g\in G$. Then $	g\in C_1\cdots C_k$ if and only if
	\begin{equation}\label{eq:structure-constant}
	\sum_{\chi \in \text{Irr}(G)}\frac{\chi(g_{1})\dots\chi(g_{k})\chi(g^{-1})}{\chi(1)^{k-1}}\neq 0.
	\end{equation}
%\sum_{\chi \in \text{Irr}(G)}\frac{\chi(g_{1})\dots\chi(g_{k})\chi(g^{-1})}{\chi(1)^{k-1}}\neq 0.
%	\] 
\end{thm}

The sum in (\ref{eq:structure-constant}) is known as the \emph{normalised structure constant} and we denote it by $ \kappa(C_1,\dots,C_k,g^{G})$. Often the summand for the trivial character is dominant, and to show that $\kappa \neq 0$ it suffices to accurately bound character ratios for $\chi \in \irr \backslash \{1_G\}$ (for example, see Proposition \ref{prop:ree} below).

%. To this end, we use the character theory of the finite simple groups to compute 
%
%\begin{thm}\label{thm:cmc}\cite[$\mathsection1$, Thm 10.1]{AH}
%	\[
%	\kappa(g_{1}^{G},\dots,g_{k}^{G},g^{G}):=\sum_{\chi \in \text{Irr}(G)}\frac{\chi(g_{1})\dots\chi(g_{k})\chi(g^{-1})}{\chi(1)^{k-1}}.
%	\] 
%\end{thm}
%Here 	$\kappa(g_{1}^{G},\dots,g_{k}^{G},g^{G})$ is known as the normalised structure constant, and by

Existing work of Guralnick and Malle \cite{GM} will also be useful in the proof of Theorem \ref{thm:MainThmLieType}. In particular, we use some cases covered by the following result.

\begin{thm}\cite[7.1]{GM} \label{thm:GM}
	Let $G$ be a rank 1 finite simple group of Lie type. Let $C$ be the conjugacy class of an element $g$ of order $o(g)>2$. Assume that one of the following holds
	\begin{enumerate}
		\item $G=PSL_2(q)$ with $q$ odd and $g$ not unipotent;
		\item $G=PSL_2(q)$ with $q$ even and $o(g)$ not divisible by $q+1$;
		\item $G={}^{2}G_{2}(3^{2n+1}), n\geq 1$ and $o(g)$ not divisible by $3$;
		\item $G={}^{2}B_{2}(2^{2n+1}), n\geq 1$; or
		\item $G=PSU_3(q)$ and $g$ is regular of order $(q^2-q+1)/d$ or $(q^2-1)/d$ with $d=\text{gcd}(3,q+1)$.
	\end{enumerate}
Then $G \backslash\{1\} \subseteq C^2.$
\end{thm}

%So this leaves us with tackling the following
%	\begin{enumerate}
%	\item $G=PSL_2(q)$ with $q$ odd and $g$ unipotent i.e. of the characteristic prime $p$; and even characteristic cases.
%	\item $G=PSL_2(q)$ with $q$ even and $o(g)$  divisible by $q+1$ if this is prime;
%	\item $G={}^{2}G_{2}(3^{2n+1}), n\geq 1$ and $o(g)=3$;
%	\item $G=PSU_3(q)$ and $g$ not regular of order $(q^2-q+1)/d$ or $(q^2-1)/d$ with $d=\text{gcd}(3,q+1)$. <-- guess this is where most of the action is.
%\end{enumerate}
%
%For the long term project of finding the $p$-width of all groups of Lie type, the small rank cases will be essential for use in inductive arguments
%We will cover the Lie type sumple groups of (twisted) rank 1. These families are
%\[
%PSL_2(q), PSU_3(q), \,{}^{2}B_{2}(2^{2n+1}), {}^{2}G_{2}(3^{2n+1}), n\geq 1;
%\]
We consider each of the families listed in Theorem \ref{thm:MainThmLieType} individually. For ease of reference we adopt the notation of \cite{AH,Orevkov, J94} for the groups $PSL_2(q), PSL^{\epsilon}(q)$ and ${}^{2}G_{2}(q^2)$, respectively.
%\subsection{$PSL_2(q)$}
\begin{prop}
	\label{prop:PSL2 p width}Let $G=PSL_{2}(q)$, $q\geq4$ and let $p$
	be an odd prime dividing $|G|$. If $q=2^{2n+1}\geq8$ and $p=3$ then $w_p(G)=3$. Otherwise $w_p(G)=2$.
	
%	\begin{enumerate}
%		\item If $q=2^{2n+1}\geq8$ then $w_p(G)=3$ if $p=3$ and $w_p(G)=2$ otherwise.
%		%\item If $q\equiv3\mod4$ then $w_p(G)=3$ if $p=2$ and $w_p(G)=2$ otherwise.
%		%\item If $q\neq2^{2n+1}$ and $q\not\equiv3\mod4$ then $w_p(G)=2$.
%	\end{enumerate}
\end{prop}
\begin{proof}
	Firstly assume that $q=2^{2n+1}\geq8$. Products of conjugacy classes in $PSL_2(q)$ are considered in \cite{AH} and in particular, the proof of \cite[$\mathsection4$, Thm. 4.2]{AH} shows that ${C}^{2}=G$ for all non-trivial conjugacy classes except
	${C}=R_{j}$, $1\leq j\leq q/2$ (see \cite{AH} for notation).
Note here that $R_{j}=(b^{j})^{G}$	for $b\in G$ an element of order $q+1$, and in fact $R_{j}^{2}$ contains all but the unipotent elements of $G$.

Now if $p\,|\,q-1$, it's clear from the conjugacy class descriptions \cite[$\mathsection4$]{AH} that there exists a class $C \subset I_p(G)$ such that $C\neq R_j$. Hence $w_p(G)=2$ by the above. If instead $3\neq p\,|\,q+1$ then $R_{(q+1)/p}\neq R_{2(q+1)/p}$ are distinct classes in $I_p(G)$. Furthermore, calculating structure constants explicitly in CHEVIE \cite{chevie} shows that  $R_{(q+1)/p}\cdot R_{2(q+1)/p}$ contains $C_2$, the unique unipotent class and so again $w_p(G)=2$. 
It  remains to consider the case where $p=3$: well as $q=2^{2n+1}$ it follows that $3\nmid q-1$ and that $R_{(q+1)/3}$ is the unique class of order $3$ elements. As remarked 
 above, $R_{(q+1)/3}^2\neq G$ and so $w_3(G)\geq 3$. But $C^3=G$ for all non-trivial classes \cite[$\mathsection4$, Thm. 4.2]{AH} and so we conclude that $w_3(G)=3$.
For other values of $q$, the proof follows in a very similar manner so we omit the details, but to summarise: Theorem \ref{thm:GM} gives a reduction to the case where $p\mid q$ and then any other necessary structure constants can be checked in CHEVIE \cite{chevie}.
\end{proof}
%\subsection{$PSU_3(q)$ and $PSL_3(q)$}
\begin{prop}\label{prop:PSU3}
Let $G=PSU_{3}(q)$, $q\geq 3$ and let $p$
be an odd prime dividing $|G|$. If $q\equiv2\mod3$ and $p=3$ then
$w_p(G)=3$, otherwise $w_p(G)=2$. 
\end{prop}
\begin{proof}
	The proposition is easily checked for $q\leq7$  using the available character tables in GAP \cite{GAP} and we therefore assume that $q\geq 8$. Now products of conjugacy classes in $G$ (and similarly in $PSL_3(q)$) are examined extensively in \cite{Orevkov}, where structure constants are computed using the generic character tables available in \cite{SF73}. We use this work to pick appropriate classes in $I_p(G)$ and bound the $p$-width.
	
	Firstly assume that $q \not\equiv 2\mod3$: here \cite[Thm 1.3 and Table 2]{Orevkov} detail exactly when $1\not\in C_a C_b C_c$ for any non-central conjugacy classes $C_i \subset SU_3(q)$. Recall that $\vert SU_3(q) \vert =q^3(q^2-q+1)(q+1)^2(q-1)$. The conjugacy classes of $SU_3(q)$ are divided into 8 families with representatives given in \cite[Table 1]{Orevkov}. For example, suppose that $p$ divides $q^2-q+1$. If $p$ also divides $q+1$ then it follows that $p=3$, a contradiction against our earlier assumption. Similarly $p \nmid q-1$ and we see that $C:=C_8^{((q^3+1)/p)}$ is a class of order $p$ elements in the $8^{th}$ family of  \cite[Table 1]{Orevkov}. We then check using \cite[Thm 1.3 and Table 2]{Orevkov} that $C^2$ contains all non-central classes of $SU_3(q)$, and as $SU_3(q)=PSU_3(q)=G$, it follows that $w_p(G)=2$. The result follows similarly when $p$ divides $q$ or $q-1$. In particular we find classes in $I_p(G)$ denoted $C_3^{(q+1)} \text{ or } C_7^{(q+1,(q^2-1)/p)}$ respectively (see \cite{Orevkov} for notation), and the squares of these classes contain all non-central elements.
	
	The final case to consider is when $3 \neq p$ divides $q+1$: here we see from \cite[Table 1]{Orevkov} that there exists a collection of prime order classes (which we denote $\mathscr{C}$) in the $4^{th}$ and $6^{th}$ families. These are
	\[
	\mathscr{C}=\{C_6^{(\alpha,2\alpha,(p-3)\alpha)},C_6^{(\alpha\beta,(p-\beta)\alpha,p\alpha)}\,\mid\,\alpha:=\frac{q+1}{p}, 1\leq \beta \leq \frac{p-1}{2} \} \text{ when }p\geq 7;
	\] 
	\[
	\mathscr{C}=\{C_4^{(\alpha\beta,\alpha\beta,(5-2\beta)\alpha)},C_6^{(\alpha,4\alpha,5\alpha)},C_6^{(2\alpha,3\alpha,5\alpha)}\,\mid\,\alpha:=\frac{q+1}{5}, 1\leq \beta \leq 4 \} \text{ when }p=5.
	\]
	Now in general $G\backslash\{1\} \nsubseteq C^2$ for a single $C\in \mathscr{C}$. However clearly $\vert \mathscr{C} \vert \geq 3$ and checking the technical conditions given in  \cite[Thm 1.3 and Table 2]{Orevkov} (these concern the eigenvalues of the three classes in question) we see that \[G\backslash\{1\}\subseteq \bigcup_{C,C'\in \mathscr{C}} CC'.\] It then follows again that $w_p(G)=2$.
	
	Lastly, assume that $q\equiv 2 \mod 3$. Provided $p\neq 3$ then the result follows by picking the same classes as above and considering their images $\tilde{C}$ in $PSU_3(q)$. If however  $p=3$ then\[\tilde{C}_2^{(0)}\notin \tilde{C}_6^{(0,\frac{q+1}{3},\frac{2(q+1)}{3})}\tilde{C}_6^{(0,\frac{q+1}{3},\frac{2(q+1)}{3})},\] by \cite[Cor. 1.8]{Orevkov}. But $\tilde{C}_6^{(0,\frac{q+1}{3},\frac{2(q+1)}{3})}$ is the unique class of order three elements and hence $w_3(G)\geq 3$ (in particular $w_3(g)\geq 3$ for $(J_2,J_1)$-unipotent elements $g\in \tilde{C}_6^{(0,\frac{q+1}{3},\frac{2(q+1)}{3})}$). But $G=(\tilde{C}_6^{(0,\frac{q+1}{3},\frac{2(q+1)}{3})})^3$ by \cite[Cor. 1.9]{Orevkov} and so the $3$-width is exactly three.
\end{proof}
\begin{rem*}
	The above result demonstrates interesting differences in the behaviour of $w_p(PSU_3(q))$ depending on whether $p$ is odd or even: if $q\equiv 2 \mod 3$ then $w_2(PSU_3(q))=3$ but otherwise $w_2(PSU_3(q))=4$ \cite[Lem. 6.9]{AJM18}.  Further families of groups where $w_2(G)\geq w_p(G)$ (for all odd $p$) are given in Proposition \ref{prop:ree} and Section \ref{sec:sporadic}. 
\end{rem*}

\begin{prop}\label{prop:PSL3}
	Let $G=PSL_3(q)$ and let $p$ be an odd prime dividing $\vert G \vert $. Then $w_p(g)=2$.
\end{prop}
\begin{proof}
	Here the proof follows almost exactly as for $PSU_3(q)$ and so we omit the details. To summarise, for each $p$ dividing $q^2+q+1, q, q+1$ and $q-1$ we identify a collection of conjugacy classes $\mathscr{C}_p$ from the $8^{th},3^{rd}, 7^{th}$ and $6^{th}$ families respectively (see \cite{Orevkov} for notation) such that $\mathscr{C}_p\subset I_p(G)$. Examining the conditions given in \cite[Thm 1.3 and Table 2]{Orevkov} we see that $G\backslash\{1\}\subseteq \bigcup_{C,C'\in \mathscr{C}_p} CC'$ for each choice of $p$, and hence $w_p(G)=2$.
	Note that despite there being a unique class of order $3$  elements (namely $\tilde{C}_6^{(0,\frac{q-1}{3},\frac{2(q-1)}{3})}$) when $q\equiv 1 \mod 3$, the square of said class does contain the unipotent class $\tilde{C}_2^{(0)}$, unlike the unitary case.
\end{proof}
%\subsection{Small Ree groups}\label{sec:small-ree}
\begin{prop}\label{prop:ree}
	Let $G={}^{2}G_{2}(q^2), q^2 \geq 27$ and let $p$ be an odd prime dividing $\vert G \vert $. Then $w_p(G)=2$.
\end{prop}

\begin{proof}
By Theorem \ref{thm:GM} it suffices to consider the case $p=3$.  Now the unipotent elements of order 3 fall into three conjugacy classes which we denote $C_3^0,C_3^{+}$ and $C_3^-$ - we refer the reader to \cite{J94} for a comprehensive description of the classes in $G$ and notation.  In particular, $C_3^+$ is a class of non-real elements of order 3, with representative $T$ such that $\vert C_G(T)\vert=2q^4$.
Using CHEVIE \cite{chevie} it is straightforward to confirm that $(C_3^{+})^2$ contains all elements of orders $2,6$ and $9$. For each remaining element of $G\backslash I_3(G)$, a conjugate is contained in $A_i \cup JA_i$, where $J$ is a representative of the single class of involutions, and $A_i, 0\leq i \leq3$ are cyclic Hall subgroups of the following orders

\[
\vert A_0 \vert =\frac{1}{2}(q^2-1),\,\,\vert A_1 \vert =\frac{1}{4}(q^2+1),\,\, \vert A_2 \vert =q^2-\sqrt{3}q+1, \,\, \vert A_3 \vert =q^2+\sqrt{3}q+1.
\] 
We check that $G\backslash I_3(G) \subset (C_3^+)^2$ using Theorem \ref{thm:cmc} and the character tables in \cite{W66}: for example, consider a cyclic subgroup $A_1$. Non-identity elements in $A_1$ are denoted by $S^a$, and if $\chi \in \irr$ is such that $\chi(T)^2\chi(S^{-a})\neq 0$ then either $\chi \in \{1_G, \xi_2, \xi_5, \xi_6, \xi_7,\xi_8\}$ or $\chi$ is one of a $(q^2-3)/6$ characters denoted $\eta_t$ or $\eta_t'$ (see \cite{W66}).
Here 
\[
\xi_5(1)=\xi_7(1)=\frac{q}{2\sqrt{3}}(q^2-1)(q^2+\sqrt{3}q+1);
\]
\[
\xi_6(1)=\xi_8(1)=\frac{q}{2\sqrt{3}}(q^2-1)(q^2-\sqrt{3}q+1);
\]
\[
\xi_2(1)=(q^6-2q^4+2q^2-1)/(q^2-1);
\]
\[
\eta_t(1)=\eta_t'(1)=q^6-2q^4+2q^2-1.
\]
Note that $\xi_5(T)=\xi_6(T)$, but $\xi_5(S^{-a})=-\xi_6(S^{-a})$ and so the corresponding summands in the structure constant formula (see Theorem \ref{thm:cmc}) sum to zero. Similarly, the summands for $\xi_7$ and $\xi_8$ sum to zero giving
\[
\kappa(T^G,T^G,(S^a)^G)=1+\sum_{\chi \in \{\xi_2\}\cup\{\eta_{t}\cup \eta_t'\}}\frac{\chi(T)^2\overline{\chi(S^a)}}{\chi(1)}.
\] 
Furthermore, $\vert \xi_2(T) \vert = \vert \eta_{t}(T) \vert= \vert \eta_{t}'(T) \vert =1$, whereas $\vert \xi_2(S^a) \vert\leq 3$ and $\vert \eta_t(S^a) \vert= \vert \eta_t'(S^a) \vert \leq 6$. These bounds give
\[
\begin{array}{ccc}
\vert \sum_{\chi \in \{\xi_2\}\cup\{\eta_{t}\cup \eta_t'\}}\frac{\chi(T)^2\overline{\chi(S^a)}}{\chi(1)} \vert & \leq & \frac{3(q^2-1)}{q^6-2q^4+2q^2-1}+\frac{q^2-3}{q^6-2q^4+2q^2-1} \\
& = & \frac{4q^2-6}{q^6-2q^4+2q^2-1} <\frac{4}{q^4-2q^2}<1. \\
\end{array}
\]
It follows that $\kappa(T^G,T^G,(S^a)^G)\neq 0$ and so $A_1 \subset (C_3^+)^2.$ The remaining cases follow similarly, giving $G=I_3(G)\cup (C_3^+)^2$ and hence $w_3(G)=2$.
\end{proof}
\begin{rem*}
	We note that each class of order 3 elements is a genuine exception to Theorem \ref{thm:GM}, as $C_3^0 \nsubseteq (C_3^{\pm})^2$ and $C_3^{\pm} \nsubseteq (C_3^{0})^2$.
\end{rem*}

\textbf{Proof of Theorem \ref{thm:MainThmLieType}} If $G={}^{2}B_{2}(2^{2n+1}), n\geq 1$ then the result is immediate from Theorem \ref{thm:GM}. The remaining cases then follow from Propositions \ref{prop:PSL2 p width} -- \ref{prop:ree}.

\section{Sporadic groups}\label{sec:sporadic}
Lastly let's consider the sporadic finite simple groups. Again, we find that the width is typically smaller for odd primes than in the case of involutions: whereas $w_2(G)=2$ for sporadic $G$ only if $G\in \{J_1, J_2\}$ \cite[Thm. 1.6]{Suleiman08},  $w_p(G)=2$ for $p\geq 3$ in almost all cases.

\begin{thm}
	\label{thm: sporadic p width}Let $(G,p)$ be a pair consisting of
	a sporadic finite simple group $G$ and an odd prime $p$, dividing
	$|G|$. Then $w_p(G)=2$, unless it is one of the
	exceptions listed in Table \ref{table:SporadicExceptions2}, where $w_p(G)=3$.
	
	\begin{table}[h] \begin{center} \caption{Sporadic groups with $p$-width 3} \label{table:SporadicExceptions2}   \renewcommand{\arraystretch}{1.1} 
			\begin{tabular}{ccc}  \hline  Group & Prime & Width three Classes \\ \hline  $HS$ & 3 & 4A, 6A\\ $Co_{2}$ & 3 & 4A\\ $Co_{3}$ & 3 & 2A\\ $Fi_{22}$ & 3,5 & 2A, 4A, 6AB, 12D ${(p=3)}$; 2A ${(p=5)}$ \\ $Fi_{23}$ & 3,5 & 2A ${(p=3,5)}$ \\ $BM$ & 3 & 2A \\ \hline \end{tabular} \end{center} \end{table}
\end{thm}
\begin{proof}
	The character tables of the sporadic simple groups are available in GAP \cite{GAP} and it is easy to calculate the relevant structure constants using Theorem \ref{thm:cmc}. We remark that in most cases $G\backslash  \{1\}\subseteq C^2$ where $C$ is (one of) the largest conjugacy classes in $I_p(G)$. Exactly which classes are not contained in $I_p(G)^2$ are given in Table \ref{table:SporadicExceptions2}.
\end{proof}

\section{Lie type groups of larger rank} \label{sec:final}
In Section \ref{sec:LieType} we saw that typically the $p$-width is two, for all odd prime divisors of a given low-rank group of Lie type. In particular, the upper bound in Theorem \ref{thm:MainThmLieType} is only reached when $p=3$. However we do not necessarily expect that this will remain true for groups of arbitrarily large rank, as the following example illustrates.

\begin{prop}\label{prop:symplectic}
	Let $p \neq l$ be primes such that $p$ is odd and $l$ is a primitive root modulo $p$. Then $G\in \{SL_{(p-1)}(l),Sp_{(p-1)}(l)\}$ has $p$-width at least 3.
\end{prop}
\begin{proof}
	Let $\Lambda$ denote the set of non-trivial $p$th roots of unity in $\ff[l^{(p-1)}]$ and let $\sigma: \Lambda \rightarrow \Lambda$ be the permutation $\sigma: \lambda\mapsto\lambda^l$.
	Recall that $l$ primitive modulo $p$ implies that $p$ divides $l^{p-1}-1$ but not $l^k-1$ for any $k<p-1$. It follows that $\sigma$ is transitive on $\Lambda=\Lambda^{-1}$ and so there exists a unique $G$-class, denoted $C$, of elements of order $p$ \cite[Sec. 3.2--3.4]{bg}. In particular $C=C^{-1}$ and elements $x \in C$ act irreducibly on $V=\ff[l]^{p-1}$.  
	Now consider a transvection $t \in G$. As transvections have order $l$, it suffices for the proposition to show that $w_p(t)\neq 2$. 
	
	Assuming that $w_p(t)=2$ for a contradiction, it follows that $t=x_1x_2$ for some $x_i \in C$. As $x_2$ is conjugate to $x_1^{-1}$ we may write this as a commutator $t=[x_1,y]$ for some $y \in G$, and applying \cite[Thm. 9]{GM2}  we see that $H=\langle x_1,y \rangle $ is contained in a Borel subgroup of $G$.
	It follows that the derived subgroup $H'$ is unipotent and in particular, $H'$ is a non-trivial $l$-group acting naturally on $V$. Considering the fixed points of this action, it is an easy consequence of the orbit-stabiliser theorem that $\vert V^{H'} \vert \equiv \vert V \vert \equiv 0 \mod l.$ But of course $V^{H'} $ contains  the zero vector, so in fact  $V^{H'} $ is a non-trivial subspace of $V$, which is invariant under $H$. This is clearly a contradiction as $H$ contains irreducible elements from $C$.
 \end{proof}

	It is clear that Proposition \ref{prop:symplectic} provides new examples of groups with $p$-width 3, for odd primes $p \neq 3$: for example $l=2$ is primitive modulo $5,11,13,19\dots$ 
	Furthermore, Proposition \ref{prop:symplectic} provides an interesting link between the $p$-width of finite groups and the famed conjecture of Artin that any integer $l$ that is neither $\pm 1$ nor a perfect square, is a primitive root modulo infinitely many primes \cite{hooley}:
	
	\begin{cor} \label{cor:artin}
		Assume that Artin's conjecture is true. Then for any prime $l$, there exists groups of Lie type $G(l)$ of arbitrarily large rank and characteristic $l$, such that $w_p(G)\geq 3$ for at least one prime divisor $p$ of $\vert G \vert$.
	\end{cor} 
\begin{proof}
	Fix $l$ and let $\{p_i\}_{i\in \mathbb{N}}$ be the primes guaranteed by Artin's conjecture. Then for $G \in \{SL_{(p_i-1)}(l), Sp_{(p_i-1)}(l)\}_{i\in \mathbb{N}},$ \[w_{p_i}(G)\geq 3\] by Proposition \ref{prop:symplectic}.
\end{proof}
\begin{rem*}
	In particular, Artin's conjecture asserts that $l=2$ is primitive modulo infinitely many primes, and so Corollary \ref{cor:artin} yields two families of simple groups of width at least 3.
\end{rem*}

\textbf{Acknowledgements:} The author is very grateful to Tim Burness for providing detailed feedback on a previous version of this paper, and to Bob Guralnick for sharing his expertise in a number of helpful conversations.

\end{document}